\documentclass[12pt]{amsart}
\usepackage[english]{babel}
\usepackage[utf8]{inputenc}
\usepackage{amsmath, amsthm, amsfonts, amssymb}
\usepackage[colorlinks=true, linkcolor=red, citecolor=green]{hyperref}
\usepackage{color}
\usepackage{mathtools}
\usepackage{graphicx}
\usepackage[english, plain]{fancyref}
\usepackage{centernot}
\usepackage{geometry}
\usepackage[shortlabels]{enumitem}
\usepackage[makeroom]{cancel}
\usepackage{lipsum}

\newcommand{\newabstract}[1]{%
  \par\bigskip
  \csname otherlanguage*\endcsname{#1}%
  \csname captions#1\endcsname
  \item[\hskip\labelsep\scshape\abstractname.]
}
\newcommand{\jump}{\vskip 2mm}

\newtheorem{proposition}{Proposition}
\newtheorem{theorem}{Theorem}
\newtheorem*{theorem*}{Theorem}
\newtheorem{lemma}{Lemma}
\newtheorem{corollary}{Corollary}
\newtheorem{ques}{Question}

\begin{document}
\title[Minimal Tjurina number of irreducible germs]{The minimal Tjurina number of irreducible germs of plane curve singularities}
\author[M. Alberich-Carramiñana]{Maria Alberich-Carramiñana}
\author[P. Almirón]{Patricio Almirón}
\author[G. Blanco]{Guillem Blanco}
\author[A. Melle-Hernández]{Alejandro Melle-Hernández}
\thanks{The first and third authors were supported by Spanish Ministerio de Ciencia, Innovaci\'{o}n y Universidades MTM2015-69135-P, Generalitat de Catalunya 2017SGR-932 projects, and they are with the Barcelona Graduate School of Mathematics (BGSMath), through the project MDM-2014-0445. The second and fourth authors were supported by Spanish Ministerio de Ciencia, Innovaci\'{o}n y Universidades MTM2016-76868-C2-1-P}

\address{Institut de Rob\`otica i Inform\`atica Industrial (IRI, CSIC-UPC)\\
 Departament\vskip 0.5mm de Matemàtiques\\
Universitat Politècnica de Catalunya · {BarcelonaTech} \\
\vskip 0.5mm Av. Diagonal 647, Barcelona 08028, Spain.}
\email{maria.alberich@upc.edu}

\address{Instituto de Matemática Interdisciplinar (IMI), Departamento de \'{A}lgebra, Geometr\'{i}a \vskip 0.5mmy Topolog\'{i}a\\
 Facultad de Ciencias Matem\'{a}ticas\\
 Universidad Complutense de Madrid\\\vskip 0.5mm
 28040, Madrid, Spain.}
\email{palmiron@ucm.es}

\address{Departament de Matemàtiques\\
Universitat Politècnica de Catalunya \vskip 0.5mm· {BarcelonaTech} \\
Av. Diagonal 647, Barcelona 08028, Spain.}
\email{guillem.blanco@upc.edu}

 \address{Instituto de Matemática Interdisciplinar (IMI), Departamento de \'{A}lgebra, Geometr\'{i}a \vskip 0.5mmy Topolog\'{i}a\\
 Facultad de Ciencias Matem\'{a}ticas\\
 Universidad Complutense de Madrid\\\vskip 0.5mm
 28040, Madrid, Spain.}
 \email{amelle@mat.ucm.es}

\selectlanguage{english}
\maketitle
\vspace{-4mm}
{\centering\footnotesize \emph{Dedicated with admiration to András Némethi on the occasion of his 60th Birthday}.\par}

\begin{abstract}
In this paper we give a positive answer to a question of Dimca and Greuel about the quotient between the Milnor and the Tjurina numbers for any irreducible germ of plane curve singularity. This result is based on a closed formula for the minimal Tjurina number of an equisingularity class in terms of the sequence of multiplicities of the strict transform along a resolution. The key points for the proof are previous results by Genzmer \cite{genzmer16}, and by Wall and Mattei \cite{wall2, mattei91}.

\end{abstract}
\section{Introduction}\label{intro}
Let \( (C, \boldsymbol{0}) \) be a germ of an isolated plane curve singularity with equation \( f \in \mathbb{C} \{x,y \} \). Associated to \( (C, \boldsymbol{0}) \) one has the Milnor number \( \mu \) and the Tjurina number \( \tau \) which are defined respectively as
\[ \mu := \dim_{\mathbb{C}} \frac{ \mathbb{C} \{x, y\} }{ (\partial f/ \partial x, \partial f/ \partial y) },\quad  \tau := \dim_{\mathbb{C}} \frac{ \mathbb{C} \{x, y\} }{ (f, \partial f/ \partial x, \partial f/ \partial y) }. \]
The Milnor number is a topological invariant of the singularity. Since topological and equisingular equivalence are the same, the Milnor number is invariant in the equisingularity class of \( (C, \boldsymbol{0}) \). On the other hand, the Tjurina number is an analytic invariant, but not in general topological. Many questions in Singularity Theory turn around knowing to what extent the topology of a germ of a singularity constrains its analytical properties.

\jump

There are different complete topological invariants, which determine the equisingularity class of the curve. For an irreducible curve $(C,\textbf{0})$ we highlight two of these invariants: the semigroup of values obtained as the intersection multiplicity of $(C,\textbf{0})$ and any other germ of curve at the origin, and the sequence of positive multiplicities of any resolution of $(C,\textbf{0})$. As it will be further detailed in section \ref{sec.moduli}, generic curves in the equisingularity class of a plane branch attain a minimal Tjurina number \(\tau_{min}\) which is in fact a topological invariant. However, a closed formula for \( \tau_{min} \) in terms of the topological invariants has not explicitly appeared before in the literature.

\jump

There have been some attempts to give a formula for \(\tau_{min}\). Most of them give a recursive expression of \(\tau_{min}\) in terms of the semigroup of the branch. For the case of one Puiseux pair \((n, m) \) with \(\gcd(n,m)=1\), Zariski proves in \cite[\S VI, pg. 107--114]{zariski-moduli} that the dimension of the generic component  of the moduli space is
\[q_{min} = \frac{(n-3)(m-3)}{2} + \Big[\frac{m}{n}\Big]-1-\mu+\tau_{min}.\]
On the other hand, Delorme in \cite{delorme78} gives a recursive formula for the dimension of the generic component in terms of the continued fraction of \(m/n\). Both results together give a recursive formula for \(\tau_{min}\) in the case of a branch with one Puiseux pair.
For a specific family of branches with two Puiseux pairs with semigroup \(\langle 2p,2q,2pq+d\rangle\) Luengo and Pfister give in \cite{luengo} the closed formula \(\tau_{min}=\mu-(p-1)(q-1)\).

\jump

 Brian\c{c}on, Granger and Maisonobe in \cite{brian} give a recursive formula for the minimal Tjurina number for the specific family of curves \( f = f_0 + g \) which are a deformation of the initial term \( f_0 = y^n - x^m \) such that \( \deg_w(f_0) < \deg_w(g) \), which are known as semi-quasi-homogeneous singularities with weights \( w = (n, m), n, m \geq 2 \).

\jump

Finally, Peraire in \cite{peraire} gives an expression for \(\tau_{min}\) in terms of the generic set of values \(\Delta_{gen}\) of the module of K\"{a}hler differentials for any irreducible plane curve singularity with any number of Puiseux pairs. The inconvenient of her approach is that  \(\Delta_{gen}\) is determined algorithmically an it provides no explicit description for \(\tau_{min}\) in terms of the equisingularity invariants.

\jump

One contribution of this work is a closed formula for \( \tau_{min} \) in the case of a plane branch, see Theorem \ref{thm:formulatau}, in terms of the sequence of multiplicities of the strict transform along a resolution. The key point for the proof of Theorem \ref{thm:formulatau} is a closed formula for the dimension \( q_{min} \) of the generic component of the moduli space of a plane branch given recently by Genzmer in \cite{genzmer16}. Our approach also relies on a formula for the dimension of the miniversal \( \mu \)--constant unfolding of \( f \) due to Wall \cite{wall2} and Mattei \cite{mattei91}.

\jump

The main goal of this paper is to give a positive answer, for any irreducible plane curve singularity, to the following question posed by Dimca and Greuel in \cite{dim}:

\begin{ques} \label{conjecture}
Is it true that \(\mu/\tau < 4/3\) for any isolated plane curve singularity?
\end{ques}

Since \(\tau\) is an upper-semicontinuous invariant, it is enough to prove the inequality for \(\mu/\tau_{min}\). Using the results explained above, in \cite{alblanc}, the second and third named authors showed that Question \ref{conjecture} has a positive answer in two cases: the case of one Puiseux pair and for semi-quasi-homogeneous singularities.

\jump

A few days after our paper was published in arXiv, another positive answer for the Dimca-Greuel question in the case of a plane branch was obtained by Genzmer and Hernandez in \cite{genzmertau}. Although the methods are rather different, the key ingredient is still the formula for the generic dimension of the moduli space obtained in \cite{genzmer16}.

\jump

\textbf{Acknowledgments.} The authors would like to thank Eduard Casas-Alvero, Manuel González-Villa and Ignacio Luengo Velasco for the helpful comments and suggestions. The authors would also like to thank the anonymous referee for helpful comments and an improvement in Corollary \ref{cor:corollary3}.

%%%%%%%%%%%%%%%%%%%%%%%%%%%%%%%%%%%%%%%%%%%%%%%%%%%% END OF THE INTRODUCTION

\section{Preliminaries}
In this section, we will introduce the basic notions and notations of resolution of singularities for an irreducible germ of plane curve singularity \((C,\textbf{0})\), also called plane branch. We refer to \cite{Caslib} and \cite{wall} for a more detailed study.

\jump

Consider \(s(x)=\sum_ja_jx^{j/n}\) the Puiseux series of the branch $(C, \boldsymbol{0})$. The set of characteristic exponents \(\{\beta_1/n,\dots,\beta_g/n\}\) is defined as
 \(\beta_1:=\min\{j\;|\;a_j\neq 0\;\text{and}\;j\notin n\mathbb{Z}\}\),
  \(\beta_{i}:=\min\{j\;|\;a_j\neq 0\;\text{and}\;j\notin e_{i-1}\mathbb{Z}\} \) with \( e_{i-1}:=\gcd(n,\beta_1,\dots,\beta_{i-1})\).
The Puiseux parameterization \((x(t),y(t))= (t^n,s(t^n))\) of the branch defines a morphism
\[\mathcal{O}_{C,\textbf{0}}:=\mathbb{C}\{x,y\}/(f)\xhookrightarrow{\qquad}\mathbb{C}\{t\},\]
and a discrete valuation \( v \) of \(\mathcal{O}_{C,\textbf{0}}\) via
\(v_C(g):=\operatorname{ord}_t(g(x(t),y(t))), g \in \mathcal{O}_{C, \boldsymbol{0}},\) which coincides with the intersection multiplicity of \( (C, \boldsymbol{0}) \) with the germ of curve defined by \( g \). Then, the semigroup of values of \( (C, \boldsymbol{0}) \) is defined as \(\Gamma(C):=\{v_C(g) \in \mathbb{Z}_{\geq 0}\ |\ g\in \mathcal{O}_{C,\textbf{0}}\setminus\{\boldsymbol{0}\}\}\).
The semigroup of \( (C, \boldsymbol{0}) \) is finitely generated, i.e. \(\Gamma(C)=\langle\overline{\beta}_0,\dots,\overline{\beta}_g\rangle\), where the \(\overline{\beta}_i\) can be expressed as \[\overline{\beta}_0:=n,\quad\overline{\beta}_{i+1}=\frac{n-e_1}{e_i}\beta_1+\frac{e_1-e_2}{e_i}\beta_2+\cdots+\frac{e_{i-1}-e_i}{e_i}\beta_i+\beta_{i+1}.\]

Let \(\pi:(S,E)\longrightarrow(\mathbb{C}^2,\mathbf{0})\) be an embedded resolution of \((C,\textbf{0})\), where \(E:=\textrm{Exc}(\pi) \) denotes the exceptional divisor of \( \pi \). The morphism \(\pi\) can be seen as the composition of point blowing up \(\pi=\pi_{\boldsymbol{0}}\circ\cdots\circ\pi_{p_k}\). Here, \(\pi_{p}\) is the blowing up of the point $p$ and \(E_{p}:=\pi_{p}^{-1}(p)\cong\mathbb{P}^{1}\) is called the exceptional divisor of $p$. The divisor \(E_{\boldsymbol{0}}\) is also called the first infinitesimal neighbourhood of the origin. The \(i\)--th neighbourhood of the origin is the set of points in the first neighbourhood of any point in the \((i-1)\)--th neighbourhood of the origin. The points in any of this neighbourhoods are called points infinitely near to the origin. The strict transform \( \widetilde{C}_{p_i} \) of \((C,\textbf{0})\) at \(p_i\) is defined as the Zariski closure of \( (\pi_{\boldsymbol{0}} \circ\cdots\circ\pi_{p_{i-1}} )^{-1}(C\setminus\{\textbf{0}\})\). The multiplicity of \((C,\textbf{0})\) at \(p\) is then \(e_p:= e_p(\widetilde{C}_p)\).

\jump

The points in an embedded resolution of an irreducible germ \((C,\textbf{0})\) are sequentially ordered by the order of neighbourhood of the origin and hence the multiplicities of \((C,\textbf{0})\) make a sequence. In the same way, the positive multiplicities of any resolution of \((C,\textbf{0})\) make a sequence. We will denote by the same symbol  \(E_{p}\) and its strict transforms on $S$ and on any intermediate blown-up surface.
An infinitely near point \(p\) is proximate to \(q\) if \(p \in E_q \). We will say that \(p\) is free (resp. satellite) if it is proximate to exactly one point (resp. two points) equal or infinitely near to the origin. The combinatorics of the resolution, i.e. the proximity relations, can be encoded as different graphs structures all of them equivalent: the Enriques diagram, the labeled dual graph, the Eisenbud-Neumann graph, the Eggers-Wall tree, etc., see \cite{Caslib} and \cite{wall} for more details.

\jump

Given a germ of plane branch \( (C, \boldsymbol{0}) \) there is a classical result which expresses the Milnor number of \( (C,\textbf{0})\) in terms of the sequence of multiplicities of any resolution of \( (C,\textbf{0}), \) see for instance \cite[6.4]{Caslib},
\begin{equation}\label{eq:milnor}
\mu  := \mu (C )= \sum_{p} e_p (e_p-1),
\end{equation}
where the summation runs on all points \( p \) equal or infinitely near to the origin. Both $\mu (C )$ and the sequence of multiplicities $ \{ e_p\}_{p}$ are topological invariants, that is, they are invariants of the equisingularity class of $(C,\textbf{0}) $. Furthermore, any finite sequence of positive $ \{ e_p \}_{ p }$ which includes all the points $p$ along a resolution of $(C,\textbf{0})$ determines the equisingularity class of $(C,\textbf{0}) $.

\section{The monomial curve and its deformations}\label{sec:monomialcurve}

Let \( \Gamma = \langle \overline{\beta}_0, \overline{\beta}_1, \dots, \overline{\beta}_g \rangle \subseteq \mathbb{Z}_{\geq 0} \) be the semigroup of a plane branch. Following Teissier in \cite{teissier-appendix}, let \( (C^\Gamma, \boldsymbol{0}) \subset (\mathbb{C}^{g+1}, \boldsymbol{0}) \) be the curve defined via the parameterization
\[ C^\Gamma : u_i = t^{\overline{\beta}_i}, \qquad 0 \leq i \leq g. \]
The germ \( (C^\Gamma, \boldsymbol{0}) \) is irreducible since \( \gcd(\overline{\beta}_0, \dots, \overline{\beta}_g) = 1 \). Moreover, \( (C^\Gamma, \boldsymbol{0}) \) is a quasi-homogeneous complete intersection, see \cite[I.2]{teissier-appendix}. The monomial curve \( C^\Gamma \) has the following important property:

\begin{theorem}[{\cite[I.1]{teissier-appendix}}]
Every branch \( (C, \boldsymbol{0}) \) with semigroup \( \Gamma \) is isomorphic to the generic fiber of a one parameter complex analytic deformation of \( (C^\Gamma, \boldsymbol{0}) \).
\end{theorem}

Therefore, every branch \( (C, \boldsymbol{0}) \) with semigroup \( \Gamma \) is analytically isomorphic to one of the fibers of the miniversal deformation \( G : (X, \boldsymbol{0}) \longrightarrow (D, \boldsymbol{0}) \) of \( C^\Gamma \). The dimension of the base \( D \) of the miniversal deformation of \( (C^\Gamma, \boldsymbol{0}) \) equals \( \mu \), the Milnor number of the equisingularity class of plane branches with semigroup \( \Gamma \), see \cite[Prop. 2.7]{teissier-appendix}.

\jump

After \cite[Thm. 3]{teissier-appendix}, we will denote by \( \tau_{-} \) the dimension of the base \( (D_\Gamma, \boldsymbol{0}) \) of the miniversal constant semigroup deformation of \( (C^\Gamma, \boldsymbol{0}) \). Let us denote by \( (C_{\boldsymbol{v}}, \boldsymbol{0}), \boldsymbol{v} \in D_\Gamma \) any fiber of the miniversal semigroup constant deformation of \( (C^\Gamma, \boldsymbol{0}) \). We will denote by \( \tau(C_{\boldsymbol{v}}) \) the dimension of the base of the miniversal deformation of \( (C_{\boldsymbol{v}}, \boldsymbol{0}) \). Similarly, we denote by \( q(C_{\boldsymbol{v}}) \) the dimension of the base of the miniversal constant semigroup deformation of the fiber \( (C_{\boldsymbol{v}}, \boldsymbol{0}) \).

\jump

By the product decomposition theorem \cite[Addendum 2.1]{teissier-appendix}, the germ of \( D_\Gamma \) at any \( \boldsymbol{v} \) is a product
\[(D_\Gamma, \boldsymbol{v}) \cong (\mathbb{C}^{\mu - \tau(C_{\boldsymbol{v}})} \times D_{\Gamma, \boldsymbol{v}}, \boldsymbol{0}), \]
where \( D_{\Gamma, \boldsymbol{v}} \) is the base of the miniversal constant semigroup deformation of \( (C_{\boldsymbol{v}}, \boldsymbol{0}) \). Thus, one has the following relation, see \cite[\S II.3.4]{teissier-appendix},
\begin{equation} \label{eq:equation0}
\tau(C_{\boldsymbol{v}}) - q(C_{\boldsymbol{v}}) = \mu - \tau_{-}.
\end{equation}

In the same way that \(\mu\) can be expressed in terms of the sequence of multiplicities along the resolution of any plane branch with semigroup \( \Gamma \), see Equation \eqref{eq:milnor}, the same is possible for \( \tau_{-} \) as we are going to show next. Assume now that \( (C_{\boldsymbol{v}}, \boldsymbol{0}), \boldsymbol{v} \in D_\Gamma \) is a plane branch and take \( f \in \mathbb{C}\{x,y\} \) any equation of \( (C_{\boldsymbol{v}}, \boldsymbol{0}) \).

\begin{proposition} \label{prop:tau_minus_unfolding}
The dimension of the miniversal \( \mu \)--constant unfolding of \( f \) equals \( \tau_{-} \).
\end{proposition}
\begin{proof}
The miniversal unfolding of the equation \( f \) has a base of dimension \( \mu \). Let us denote by \( \vartheta \) the dimension of the base of the miniversal \( \mu \)--constant unfolding of \( f \). The codimension of the \( \mu \)--constant stratum is then \( \mu - \vartheta \). Now, since the miniversal unfolding of \( f \) is a versal deformation of \( (C_{\boldsymbol{v}}, \boldsymbol{0}) \), the codimension of the \( \mu \)--constant strata of both deformations coincide. The curve \( (C_{\boldsymbol{v}}, \boldsymbol{0}) \) being plane implies that \( \mu \)--constant is equivalent to constant semigroup, \cite{le}, and hence, \( \tau(C_{\boldsymbol{v}}) - q(C_{\boldsymbol{v}}) = \mu - \vartheta \). Finally, by Equation \eqref{eq:equation0}, \( \vartheta \) equals \( \tau_{-} \).
\end{proof}

Finally, the dimension of the \(\mu\)--constant stratum of the miniversal unfolding of any reduced \( f \in \mathbb{C}\{x,y\} \), which we will also denote by \( \tau_{-} \) after Proposition \ref{prop:tau_minus_unfolding}, is computed by Mattei \cite{mattei91} and Wall \cite{wall2} in terms of the sequence of multiplicities of the strict transform along an embedded resolution of the germ \( (C, \boldsymbol{0}) \).

\begin{theorem}[{\cite[Thm. 4.2.1]{mattei91}, \cite[Thm. 8.1]{wall2}}]\label{thm:tau_minus}
The dimension of the \(\mu\)--constant stratum of the miniversal unfolding of \( f \) equals
\[ \tau_{-} = \sum_{p} \frac{(e'_p - 2)(e'_p - 3)}{2}, \]
where the summation runs on all points \( p \) equal or infinitely near to the origin and
\begin{enumerate}[(a)]
\item \( e'_p := e_p \) if \( p \) is the origin,
\item \( e'_p := e_p + 1 \) if \( p \) is free and \(e_p>0\),
\item \( e'_p := e_p + 2 \) if \( p \) is satellite and \(e_p>0\),
\item \( e'_p := 2 \) otherwise.
\end{enumerate}
\end{theorem}

{\setlength{\parindent}{0cm}
\textbf{Remark.} It might be worth noticing that the quantity on the left (or right) hand-side of Equation \eqref{eq:equation0} coincides with the codimension \( \tau^{es} \) of the equisingularity ideal \( I^{es} \) of \( (C_{\boldsymbol{v}}, \boldsymbol{0}) \), see for instance \cite{greuel-book} and the references therein. Similarly, it can be seen that \(\tau_{-}\) also equals the modality of \(f\) in the sense of Wall \cite{wall2}.

\section{The dimension of the generic component of the moduli space} \label{sec.moduli}

Following the notations from the last section, let us begin this section by defining the moduli space of plane branches with semigroup \( \Gamma \). Analytically equivalence of germs induces an equivalence relation \( \thicksim \) in \( D_\Gamma \). The topological space \( D_\Gamma / \thicksim \), with the quotient topology, will be denoted by \( \widetilde{M}_\Gamma \) and it is called the \emph{moduli space} associated to the semigroup \( \Gamma \). Let \( m : D_\Gamma \longrightarrow \widetilde{M}_\Gamma \) be the natural projection and let \( D^{(2)}_\Gamma \) be the following subset of \( D_\Gamma \)
\[ D^{(2)}_\Gamma := \{ \boldsymbol{v} \in D_\Gamma \ |\ (G^{-1}(\boldsymbol{v}), \boldsymbol{0})\ \textrm{is a plane branch} \}.  \]
Then, Teissier proves in \cite{teissier-appendix} that \( D^{(2)}_\Gamma \) is an analytic open dense subset of \( D_\Gamma \) and that \( m(D^{(2)}_\Gamma) \) is the moduli space \( M_\Gamma \) of plane branches with semigroup \( \Gamma \) in the sense of Zariski \cite{zariski-moduli}. Moreover, \( \widetilde{M}_\Gamma = M_\Gamma \) if and only if \( \Gamma \) is generated by two elements.

\jump

Following \cite{zariski-moduli}, we define the generic curve \( C_v \) of the moduli space \( \widetilde{M}_\Gamma \) as the fiber corresponding to the generic point \( v \) in the base space \( D_\Gamma \) of the miniversal semigroup constant deformation. By the discussion in the previous paragraph, \( C_v \) is a plane branch and coincides with the generic fiber defined in \( M_\Gamma \). Since \( \tau(C_v) \) coincides with the dimension of the Tjurina algebra of \( (C_{\boldsymbol{v}}, \boldsymbol{0}) \) and it is upper-semicontinuous, we can define \( \tau_{min} := \tau(C_v) \). Similarly, after Equation \eqref{eq:equation0}, it makes sense to define \( q_{min} := q(C_v) \).

\jump

The quantity \( q_{min} \) is the dimension of the generic component of the moduli space \( \widetilde{M}_\Gamma \) of branches with semigroup \( \Gamma \), see \cite[Thm. 6]{teissier-appendix}. Therefore, after applying Equation \eqref{eq:equation0} to the generic branch \( (C_v, \boldsymbol{0}) \), see \cite[\S II.3.6]{teissier-appendix},
\begin{equation} \label{eq:equation1}
  \tau_{min} - q_{min} = \mu - \tau_{-},
\end{equation}
and, after Theorem \ref{thm:tau_minus}, computing \( q_{min} \) is equivalent to computing \( \tau_{min} \).

\jump

Recently, Genzmer in \cite{genzmer16}, computed the dimension of the generic component of the moduli space \( q_{min} \) for any plane branch \( (C, \boldsymbol{0}) \) in terms of the sequence of multiplicities of the strict transform along the minimal embedded resolution of the germ \( (C, \boldsymbol{0}) \). With the same notations from Theorem \ref{thm:tau_minus},

\begin{theorem}[\cite{genzmer16}] \label{thm:dimension-generic}
The dimension of the generic component \( q_{min} \) of the moduli space of plane branches with semigroup \(\Gamma\) equals
\[ q_{min} = \sum_{p} \sigma(e'_p), \]
where the summation runs on all points \( p \) equal or infinitely near to the origin and
\[
 \sigma(k) := \left\{
\begin{array}{ll}
\displaystyle \frac{(k-2)(k-4)}{4},\ \textrm{if \( k \) is even}, \vspace{3mm} \\
\displaystyle \frac{(k-3)^2}{4},\ \textrm{if \( k \) is odd}.
\end{array} \right.
\]
\end{theorem}

\section{The minimal Tjurina number of an equisingularity class}

In this section, we present a closed formula for the minimal Tjurina number $\tau_{min}$ of an equisingularity class of a plane branch \( (C, \boldsymbol{0}) \) in terms of the sequence of multiplicities $ \{ e_p \}_{ p }$. From the formula for the \( \tau_{min} \), some consequences will be inferred, including the main result of this paper which is a positive answer to Dimca and Greuel's question on the quotient \(\mu/\tau\) of any plane branch.

\begin{theorem}\label{thm:formulatau}
For any equisingular class of germs of irreducible plane curve singularity,
\begin{equation*}
\begin{split}
  \tau_{min} = \sigma(n) + \frac{n^2+3n-6}{2} + & \sum_{p\, \textrm{free}} \frac{(e_p - 1)(e_p + 2) + 2\sigma(e_p + 1)}{2} \\
  + & \sum_{p\, \textrm{sat.}} \frac{e_p(e_p - 1) + 2 \sigma(e_p + 2)}{2},
\end{split}
\end{equation*}
where the summation runs on all points \( p \) equal or infinitely near to the origin.
\end{theorem}
\begin{proof}
Using Equation \eqref{eq:equation1} we have \( \tau_{min} = q_{min} + \mu - \tau_{-},\) from Theorems \ref{thm:tau_minus} and \ref{thm:dimension-generic} the claimed formula follows.
\end{proof}

 A fortiori, we can see from Theorem \ref{thm:formulatau} that the \( \tau_{min} \) of an equisingularity class depends only on the minimal resolution of \( (C, \boldsymbol{0}) \) and not on the minimal embedded resolution. Furthermore, the formula works for any resolution of \( (C, \boldsymbol{0}) \), minimal or not.
\jump

This formula for \(\tau_{min}\) enables us to give a positive answer to Dimca and Greuel's question, see Question \ref{conjecture}, in the case of any plane branch. Before proving this, we need the following property of the sequence of multiplicities.

\begin{lemma} \label{lemma:1}
For any plane branch singularity of multiplicity $n$,
\[ \sum_{p\, \textrm{sat.}} e_p = n - 1, \]
where the summation runs on all satellite infinitely near points \( p \).
\end{lemma}

\begin{proof}
Consider the finite sequence of positive multiplicities $ \{ e_p \}_{ p }$ along the minimal embedded resolution of the plane branch.
From Enriques' theorem \cite[Thm. 5.5.1]{Caslib} one has that \( n + \sum_{p\, \textrm{free}} e_p = \beta_g \), where \(\beta_g/n\) is the last characteristic exponent. On the other hand, from \cite[Ex. 5.6]{Caslib}, \( \sum_{p} e_p = \beta_g + n - 1, \) where this summation runs on all points \( p \) equal or infinitely near to the origin.
Since all the satellite points for which $e_p$ is positive are included in the sequence of points blown-up in the minimal embedded resolution, the result follows.
\end{proof}

Finally, we get the announced positive answer to Dimca and Greuel's question as a Corollary of Theorem \ref{thm:formulatau}.

\begin{corollary}\label{cor:4/3}
For any plane branch singularity, \[ \frac{\mu}{\tau} < \frac{4}{3}. \]
\end{corollary}
\begin{proof}
It is enough to prove the inequality for the \( \tau_{min} \) of each equisingularity class of plane branches. We will show that \(4\tau_{min}-3\mu>0\). From Theorem \ref{thm:formulatau} and Equation \eqref{eq:milnor} we have that
\begin{equation*}
\begin{split}
  4\tau_{min}-3\mu  =4 \sigma(n) - n^2 + 9n - 12  &  + \sum_{p\,\textrm{free}} \left( 4 \sigma(e_p + 1) - (e_p-1)(e_p-4) \right) \\
 &+ \sum_{p\, \textrm{sat.}} \left( 4 \sigma(e_p + 2) - e_p(e_p - 1) \right).
\end{split}
\end{equation*}
Now, since \( \sigma(k) \geq (k-2)(k-4)/4 \),
\[ 4 \tau_{min} - 3\mu \geq 3n -4 + \sum_{\substack{p\, \textrm{free}\\ e_p > 0}} (e_p - 1) - \sum_{p\, \textrm{sat}.} e_p. \]
Finally, using Lemma \ref{lemma:1}
\[ 4 \tau_{min} - 3\mu \geq 2n -3 + \sum_{\substack{p\, \textrm{free}\\ e_p >0}} (e_p - 1) > 0, \]
and the result follows, since \(n\geq 2\).
\end{proof}
As a direct consequence of Theorem \ref{thm:formulatau} we also obtain the following lower bound for \(\tau\):
\begin{corollary} \label{cor:corollary3}
For any plane branch,
\[
\begin{split}
  \tau \geq  \frac{3 n^2}{4}-1 \quad \textnormal{if} \quad n\ \textnormal{is even},\\
  \tau \geq  \frac{3}{4}(n^2-1) \quad \textnormal{if} \quad n\ \textnormal{is odd}.
\end{split}
\]
\end{corollary}
The bound in Corollary \ref{cor:corollary3} is sharp, as one can easily check for generic curves in the equisingularity class of the singularities \( y^n -x^{n+1} = 0, \) i.e. the minimal Tjurina number in the equisingularity class of the singularities \( y^n -x^{n+1} = 0 \) coincides with the bound of Corollary \ref{cor:corollary3}. In fact, from Theorem \ref{thm:formulatau} one can see that these are the only topological types of singular plane branches for which the bound is reached. 

\jump

 Finally, we are going to analyze the set of values of K\"{a}hler differentials that we will denote by \(\Delta(C):=v(\mathcal{O}_{C,\textbf{0}}+\mathcal{O}_{C,\textbf{0}}(dy/dx))\) . It is easy to check that \(\Gamma(C)\) acts on addition in \(\Delta(C)\), hence \(\Delta(C)\) is a \(\Gamma(C)-\)monoid. It is proven by Delorme in \cite{delorme78} that \(\Delta(C)\) is an analytical invariant of the singularity. The set of K\"{a}hler differentials is used to define some of the main numerical analytical invariants of plane curves singularities, for example the so called Zariski exponent. Peraire's procedure to compute \(\tau_{min}\) is closely related to the invariant \(\Delta(C)\). For the generic curve \((C_{v},\textbf{0})\) of the moduli space of plane branches with semigroup \( \Gamma \), see Section \ref{sec.moduli}, Peraire gives an algorithm to compute, from the semigroup of values \(\Gamma\), a set of generators for \(\Delta_{gen}:=\Delta(C_v).\) From this, Peraire gives the following formula for the \( \tau_{min} \),
 \begin{equation}
 \label{eqn:Peraire}
 \tau_{min}=\frac{\mu}{2}+n-1+|\mathbb{N}\setminus\Delta_{gen}|,
 \end{equation}
 where \( |\cdot| \) denotes the cardinal.
 \jump

 From the formula from Equation \eqref{eqn:Peraire} and Theorem \ref{thm:formulatau}, we can infer the following:

 \begin{corollary}\label{cor:gaps}
 For any equisingularity class of plane branch singularities the cardinal of the gaps of $\Delta_{gen}$ is
  \[|\mathbb{N}\setminus\Delta_{gen}| =\sigma(n) + n-2 +  \sum_{p\, \textrm{free}} ((e_p - 1) + \sigma(e_p + 1))   +  \sum_{p\, \textrm{sat.}} \sigma(e_p + 2),
 \]
  where the summation runs on all points \( p \) equal or infinitely near to the origin.
   %in a minimal embedded resolution.
 \end{corollary}

\jump
\jump


\begin{thebibliography}{99}

\bibitem{alblanc} P.~Almir\'{o}n, G.~Blanco, \textit{A note on a question of Dimca and Greuel},  C. R. Math. Acad. Sci. Paris,  Ser. I \textbf{357} (2019), 205--208.

\bibitem{brian} J.~Brian\c{c}on, M.~Granger, Ph.~Maisonobe, \textit{Le nombre de modules du germe de courbe plane \(x^a+y^b=0\)}, Math. Ann. \textbf{279} (1988), 535--551.

\bibitem{Caslib} E.~Casas-Alvero, \textit{Singularities of plane curves}, London Math. Soc. Lecture Note Ser. \textbf{276}, Cambridge
Univ. Press, Cambridge, 2000.

\bibitem{delorme78} C.~Delorme, \textit{Sur les modules des singularit\'{e}s des courbes planes}, Bull. Soc. Math. France \textbf{106} (1978), 417--446.

\bibitem{dim} A.~Dimca, G.-M.~Greuel, \textit{On 1-forms on isolated complete intersection on curve singularities}, J. of Singul. \textbf{18} (2018), 114--118.

\bibitem{genzmer16} Y.~Genzmer, \textit{Dimension of the moduli space of a curve in the complex plane}, Preprint in: \href{https://arxiv.org/pdf/1610.05998.pdf}{arXiv:1610.05998} (2016).

 \bibitem{genzmertau} Y. Genzmer, M. E. Hernandes, \textit{On the Saito’s basis and the Tjurina Number for Plane Branches}, Preprint in: \href{https://arxiv.org/pdf/1904.03645.pdf}{arXiv:1904.03645} (2019).

\bibitem{greuel-book} G.-M.~Greuel, C.~Lossen, E.~Shustin, \textit{Introduction to Singularities and Deformations}, Springer Monographs in Mathematics, Berlin, 2007.

\bibitem{le} L.~D.~Tráng, C.~P.~ Ramanujam, \textit{The invariance of Milnor's number implies the invariance of the topological type}, Amer. J. Math. \textbf{98} (1976), no. 1, 67--78.

\bibitem{luengo} I.~Luengo, G.~Pfister, \textit{Normal forms and moduli spaces of curve singularities with semigroup \(\langle 2p,2q,2pq+d\rangle\)}, Compos. Math. \textbf{76} (1990), no.~1--2, 247--264.

\bibitem{mattei91} J.~F.~Mattei, \textit{Modules de feuilletages holomorphes singuliers: {I} équisingularité}, Invent. Math. \textbf{103} (1991), no. 2, 297--325.

\bibitem{peraire} R.~Peraire, \textit{Tjurina number of a generic irreducible curve singularity}, J. Algebra \textbf{196} (1997), no.1, 114--157.

\bibitem{wall2}  C.~T.~C.~Wall, \textit{Notes on the classification of singularities}, Proc. London Math. Soc. \textbf{48} (1984), no.3, 461--513.

\bibitem{wall} C.~T.~C.~Wall, \textit{Singular points of plane curves}, London Math. Soc. Students Texts \textbf{63}, Cambridge
Univ. Press, Cambridge, 2004.

\bibitem{teissier-appendix} B.~Teissier, \textit{Appendix}, in \cite{zariski-moduli}, 1986.

\bibitem{zariski-moduli} O.~Zariski, \textit{Le probl\'{e}me des modules pour les branches planes}, Hermann, Paris, 1986.

\end{thebibliography}
\end{document}